\documentclass[11pt]{article}

\usepackage[utf8]{inputenc}
\usepackage[T1]{fontenc}
\usepackage{lmodern}
\usepackage[letterpaper,margin=1in]{geometry}
\usepackage{amsmath,amssymb,amsthm}
\usepackage{microtype}
\usepackage{xcolor}
\usepackage{tikz-cd}
\usepackage[colorlinks=true,linkcolor=blue,citecolor=blue,urlcolor=blue]{hyperref}

\newtheorem{theorem}{Theorem}[section]
\newtheorem{lemma}[theorem]{Lemma}

\newtheorem{corollary}[theorem]{Corollary}

\theoremstyle{definition}
\newtheorem{definition}[theorem]{Definition}
\newtheorem{remark}[theorem]{Remark}

\makeatletter
\renewcommand*\@fnsymbol[1]{\ensuremath{%
  \ifcase#1\or
    *\or
    \star\or     
    \ddagger\or
    \mathsection\or
    \mathparagraph\or
    \|\or
    **\or
    \star\star\or
    \ddagger\ddagger
  \else
    \@ctrerr
  \fi}}
\makeatother

\title{A Colorful Extension of VC-dimension and Geometric Applications}
\author{Chaya Keller\thanks{School of Computer Science, Ariel University, Israel. \texttt{chayak@ariel.ac.il}.}
	\mbox{ }
	and Shakhar Smorodinsky\thanks{Department of Computer Science, Ben-Gurion University of the NEGEV, Be’er Sheva 84105, Israel. 
		\texttt{shakhar@bgu.ac.il}
}}
\date{}

\begin{document}

\maketitle

\begin{abstract}
The VC-dimension is a fundamental measure of the complexity of a set system. In this paper, we introduce and study a colorful variant of VC-dimension that captures the behavior of set systems on colored ground sets.

By studying this new notion, we obtain a variety of geometric results. First, we prove that separable abstract convexity spaces with Radon number $D$ admit a Tverberg theorem with Tverberg number $O(D^2 r \log r)$. This bound significantly improves the $O(Dr^2\log r)$ bound of Alon and Smorodinsky from SODA'26 and is the first quasi-linear bound in $r$, in which the dependence on $D$ is not super-exponential. Second, we prove the first 
colorful $k$-wise Tverberg theorem for separable abstract convexity spaces. Using this theorem, we obtain a colorful selection lemma with $O(D^3)$ colors, an uncolored selection lemma for subsets of size $O(D^3)$, a weak $\varepsilon$-net theorem with nets of size $O_D(\varepsilon^{-O(D^3)})$, and a $(p,q)$-theorem with exponent of $\mathrm{poly}(D)$. All these quantitative bounds are significantly better than the best previously known general bounds for abstract convexity spaces. Finally, we extend our method to obtain a colorful Tverberg theorem for unions of convex sets, generalizing the uncolored theorem of Alon and Smorodinsky (SODA'26). 
\end{abstract}

\section{Introduction}

The VC-dimension is a fundamental measure of the complexity of a set system, with applications in combinatorics, discrete geometry, sampling, and learning theory. In its classical form, the VC-dimension is inherently uncolored: it measures which subsets of a ground set can be realized as traces of hyperedges, without taking into account additional structure such as colors, groups, etc.
Various natural combinatorial and geometric settings do carry such extra structure, and the relevant configurations are often rainbow configurations rather than arbitrary subsets. This raises a basic question: what is the right VC-type notion for colored set systems, and what structural consequences follow from bounded ordinary VC-dimension in such settings?

In this paper we develop a colorful VC framework for this purpose. We introduce $k$-wise colorful and rainbow shattering parameters for hypergraphs and prove a general upper bound for these parameters in terms of the ordinary VC-dimension. This can be viewed as a colorful analogue of the Perles--Sauer--Shelah lemma: bounded VC-dimension limits the number of rainbow assignments that can be forced across color classes. 

We apply this framework to separable abstract convexity spaces. An \emph{abstract convexity space} is a pair $(X,\mathcal C)$, where $X$ is a nonempty set and $\mathcal C$ is a family of subsets of $X$ such that $\emptyset,X\in\mathcal C$ and the family $\mathcal C$ is closed under intersections and under nested unions: if $\mathcal A\subseteq\mathcal C$ is nonempty and totally ordered by inclusion, then $\bigcup_{A\in\mathcal A}A\in\mathcal C$.\footnote{Closure under nested unions is included in the definition used by Holmsen and Pat\'akov\'a~\cite{HolmsenP24}. In the present paper it is needed only in order to apply their fractional Helly theorem, and hence only for Theorems~\ref{thm:selection_lemma}--\ref{thm:pq-separable}; the results through Theorem~\ref{thm:improved-tverberg} use only closure under intersections.} The sets in $\mathcal C$ are called \emph{convex sets}. For a subset $A\subseteq X$, its convex hull is defined by $\operatorname{conv}(A)=\bigcap\{C\in\mathcal C:A\subseteq C\}$. Thus, $\operatorname{conv}(A)$ is the smallest convex set containing $A$.

A subset $H\subseteq X$ is called a \emph{halfspace} if $H\in\mathcal C$ and $X\setminus H\in\mathcal C$. We say that the convexity space $(X,\mathcal C)$ is \emph{separable} if 
for every two disjoint convex sets $A,B\in\mathcal C$, there is a halfspace $H$ such that $A\subseteq H$ and $B\subseteq X\setminus H$, and for every convex set $C\in\mathcal C$ and every point $x\notin C$, there is a halfspace $H$ such that $C\subseteq H$ and $x\notin H$.\footnote{The first requirement is the separability notion used by Alon and Smorodinsky~\cite{AlonS26}, while the second is the separability notion used by Holmsen and Pat\'akov\'a~\cite{HolmsenP24}. The results through Theorem~\ref{thm:improved-tverberg} require only the first separation property. Theorems~\ref{thm:selection_lemma}--\ref{thm:pq-separable} also use the second property through the fractional Helly theorem of Holmsen and Pat\'akov\'a. For simplicity, throughout the paper we use the term ``separable'' for spaces satisfying both requirements.}

For $r\ge 2$, the \emph{$r$-th Tverberg number} of $(X,\mathcal C)$, denoted by $\operatorname{Tv}_{\mathcal C}(r)$, is the smallest integer $N$, if it exists, such that every set $P\subseteq X$ of size $N$ admits a partition $P=P_1\dot\cup\cdots\dot\cup P_r$ with $\bigcap_{i=1}^r\operatorname{conv}(P_i)\neq\emptyset$. The case $r=2$ is the Radon number of the space. Throughout the paper we denote by $D=\operatorname{Tv}_{\mathcal C}(2)$ the \emph{Radon number} of the space, and by $r$ the number of parts in a Tverberg partition.

\subsection{Background}

\paragraph{VC dimension and shattering.}
Let $H=(V,E)$ be a hypergraph. In what follows, we sometimes refer to elements of $E$ as ranges or hyperedges. For a set $T\subseteq V$ and a hyperedge $e\in E$, the set $T\cap e$ is called the {\em trace} of $e$ on $T$. A set $S\subseteq V$ is said to be {\em shattered} by $H$ if every subset of $S$ occurs as a trace, that is, if $\{S\cap e:e\in E\}=2^S$. The {\em VC dimension} of $H$ is the supremum on the sizes of shattered sets, 
\(
\operatorname{VCdim}(H)=\sup \left\{|S|:S\subseteq V \text{ is finite and } \{S\cap e:e\in E\}=2^S \right\}.
\)
 We shall use the Perles--Sauer--Shelah lemma~\cite{Sau72,Shelah72}: if $\operatorname{VCdim}(H)\le v$, then for every finite set $T\subseteq V$ with $|T|=m$, one has $\pi_H(T):=|\{T\cap e:e\in E\}|\le \sum_{i=0}^v {m\choose i}=O(m^v)$. The function $\pi_H(m)=\mbox{max} \{\pi_H(T):|T|=m,T \subset V\}$ is called the {\em shattering function} of $H$.

In SODA'26, Alon and Smorodinsky~\cite{AlonS26} proposed a generalization of shattering which they coined `$r$-shattering'. For $r \geq 2$, $S \subset V$ is $r$-shattered by $H$ if for every partition $S=S_1 \cup \ldots \cup S_r$, there exist hyperedges $e_1,\ldots,e_r \in E$ such that $(\cap_{i=1}^r e_i) \cap S = \emptyset$, and for each $i$, $S_i \subset e_i$. The $r$-VC dimension of $H$ is the maximum size of a set that is $r$-shattered by $H$. Note that for $r=2$, this notion reduces to the standard notion of shattering. 
They used the Perles--Sauer--Shelah lemma to 
show that if $\mathrm{VCdim}(H) \leq v$ then the $r$-VC dimension of $H$ is $O(vr^2 \log r)$, and applied this bound to obtain a Tverberg theorem for unions of convex sets with a near-optimal bound on the Tverberg number. In a recent follow-up paper~\cite{CGSWX25}, Chen et al.~provided near matching lower bounds and further explored the applications of the technique of~\cite{AlonS26} to Tverberg-type results for unions of convex sets.

The new notion of $r$-shattering presented in~\cite{AlonS26} lies in a stark difference with the rich theory of generalizations of the VC dimension, 
such as the Natarajan~\cite{Natarajan89}, the DS~\cite{DanielyS14}, and the Steele~\cite{Steele78} dimensions.
Indeed, as the previously studied generalizations were motivated by questions from machine learning  (e.g.,~\cite{BrukhimCDMY22,DanielyS14,Natarajan89}), mathematical logic (e.g.,~\cite{ChernikovPT19}), etc., 
they focus on shattering by \emph{a single hyperedge}, where the `shattered domain' is generalized from $\{0,1\}^S$ to $\{0,1,\ldots,m\}^S$ 
and the intersections $e \cap S$ are required to cover some portion of $\{0,1,\ldots,m\}^S$, such as a `copy' of $\{0,1\}^S$ for the Natarajan dimension. In contrast, the geometric-oriented definition of~\cite{AlonS26} leaves the shattered domain unchanged and considers the patterns in which $r$-tuples of hyperedges intersect it simultaneously.

The method of Alon and Smorodinsky~\cite{AlonS26} shows that VC-dimension bounds can be used not only to count traces, but also to control more structured shattering phenomena that arise in Radon and Tverberg-type problems. Our contribution is to refine this approach by introducing colorful and $k$-wise rainbow shattering notions. These variants are more flexible than the original notion in~\cite{AlonS26}: they lead to a broad range of applications, including new colorful Tverberg-type theorems and improved quantitative bounds for Tverberg theorems, weak epsilon-nets, and $(p,q)$-theorems in separable convexity spaces.

\paragraph{Radon and Tverberg numbers in abstract convexity spaces.}
Radon's theorem states that every set of $d+2$ points in $\mathbb R^d$ can be partitioned into two parts whose convex hulls intersect. Tverberg's classical theorem extends this from two parts to $r$ parts: every set of $(d+1)(r-1)+1$ points in $\mathbb R^d$ admits a partition into $r$ parts whose convex hulls have a common point. 
In terms of the Tverberg number $\operatorname{Tv}_{\mathcal C}(r)$ defined above, Radon's theorem says that $\operatorname{Tv}_{\mathbb{R}^d}(2) \leq d+2$, while Tverberg's theorem says that $\operatorname{Tv}_{\mathbb R^d}(r)\le (d+1)(r-1)+1$. A central problem in abstract convexity is to understand to what extent such Tverberg-type theorems are forced by Radon-type theorems.

 In its strongest form, this question was formulated as Eckhoff's partition conjecture~\cite{Eckhoff00}, which states that if the Radon number of an abstract convexity space is $D$, then $\operatorname{Tv}_{\mathcal C}(r)=(D-1)(r-1)+1$ (the same bound as in the Euclidean case, in which $D=d+2$). This conjecture was disproved by Bukh~\cite{Bukh10}. However, the known counterexamples and their natural generalizations miss the conjectured value only by an additive term of order $O(r)$; see Holmsen~\cite{Holmsen24}. This motivates the following weaker form, known as the \emph{weak Eckhoff conjecture}: $\operatorname{Tv}_{\mathcal C}(r)=O(Dr)$. The weak Eckhoff conjecture remains a central open problem in abstract convexity; see, for instance,~\cite{BK22,BaranyS18,Holmsen24,Palvolgyi22}.

The best general upper bound currently known is due to P\'alv\"olgyi.

\begin{theorem}[P\'alv\"olgyi~\cite{Palvolgyi22}]
Let $(X,\mathcal C)$ be an abstract convexity space with Radon number $D$. Then  $\operatorname{Tv}_{\mathcal C}(r)\le O\!\left(D^{D^{D^{\log D}}}r\right)$.
\end{theorem}

Thus, P\'alv\"olgyi's theorem proves the qualitative content of the weak Eckhoff conjecture, namely linear growth in $r$, but with a very large dependence on the Radon number $D$.
For separable convexity spaces, Alon and Smorodinsky obtained a different bound, using their new notion of $r$-shattering described above and a Radon-to-Tverberg leveraging argument based on VC-dimension considerations.

\begin{theorem}[Alon--Smorodinsky~\cite{AlonS26}]
Let $(X,\mathcal C)$ be a separable abstract convexity space with Radon number $D$. Then $\operatorname{Tv}_{\mathcal C}(r)\le O\!\left(Dr^2\log r\right)$.
\end{theorem}

This bound has a much better dependence on $D$, but it is quadratic in the number of parts $r$. 

\paragraph{Relaxed and colored Tverberg theorems.}
A related line of work studies relaxed Tverberg-type partitions. Reay~\cite{Reay79} considered the following relaxation in the Euclidean setting. If $T(d,r,k)$ denotes the smallest integer $n$ that guarantees a partition of any point set in $\mathbb R^d$ of cardinality $n$ into $r$ parts such that the convex hulls of every $k$ parts have nonempty intersection, then the case $k=r$ is the ordinary Tverberg number. Reay conjectured that this relaxation does not reduce the required number of points, namely that $T(d,r,k)=T(d,r,r)$ for all $2\le k\le r$; see also \cite{PerlesSigron16}. This motivates the study of $k$-wise Tverberg-type theorems also in abstract convexity spaces. 

Colored versions of Tverberg's theorem have been studied extensively since the work of B\'ar\'any and Larman~\cite{BaranyLarman92}. In the Euclidean setting, the classical colored Tverberg problem asks whether, given $d+1$ color classes in $\mathbb R^d$, each of size $r$, one can partition the points into $r$ pairwise disjoint rainbow sets whose convex hulls have a common point. The planar case was proved by B\'ar\'any and Larman, who conjectured that the answer is positive in all dimensions. Further results and optimal or near-optimal bounds were obtained by topological methods, in particular by \v{Z}ivaljevi\'c and Vre\'cica~\cite{ZivaljevicVrecica92} and by Blagojevi\'c, Matschke and Ziegler~\cite{BlagojevicMatschkeZiegler15}; see also Sober\'on~\cite{Soberon15} for related variants with equal coefficients and tolerance. 

\paragraph{Tverberg-type theorems for unions of convex sets.}
Another motivation for the shattering approach comes from Tverberg-type problems for unions of convex sets. A set in $\mathbb R^d$ is called $s$-convex if it is the union of at most $s$ convex sets. In the 1970s, Kalai asked whether Radon's theorem admits an analogue in which convex sets are replaced by unions of a bounded number of convex sets. More precisely, one asks for the smallest integer $f(d,s)$ such that any set $P\subseteq\mathbb R^d$ of size $f(d,s)$ admits a partition $P=A\dot\cup B$ with the property that any $s$-convex set containing $A$ intersects any $s$-convex set containing $B$; see \cite[Problem 6.6]{BK22} and \cite[Problem 14]{Kalai17}. Alon and Smorodinsky~\cite{AlonS26} recently settled this problem, and further used their extended VC-dimension method to prove Tverberg-type analogues for unions of convex sets. 

\paragraph{Selection lemmas, fractional Helly theorems, and weak $\varepsilon$-nets.}
Selection lemmas form another classical consequence of Tverberg-type phenomena\footnote{Selection lemmas are also closely related to high-dimensional expanders, see \cite{Gro10}.}. The classical selection lemma in $\mathbb R^d$, due to Bárány \cite{Barany82}, states that for every finite set $P\subset\mathbb R^d$ of size $n$, there exists a point $x\in\mathbb R^d$ contained in the convex hulls of at least a positive fraction, depending only on $d$, of all $(d+1)$-element subsets of $P$.
Alon, Kalai, Matou\v{s}ek and Meshulam~\cite{AKMM02} developed a general abstract framework in which Tverberg-type statements, selection lemmas, weak $\varepsilon$-nets and $(p,q)$-theorems are derived from fractional Helly-type assumptions.

Helly's theorem states that for a finite family of convex sets in $\mathbb R^d$, if every $d+1$ members have a nonempty intersection, then the whole family has a nonempty intersection. More generally, the Helly number of a convexity space is the smallest integer $h$, if it exists, such that for every finite family $\mathcal F$ of convex sets, if every $h$ members of $\mathcal F$ have a nonempty intersection, then $\bigcap_{F\in\mathcal F}F\neq\emptyset$. A fractional Helly theorem is a quantitative strengthening of Helly's theorem. We say that a convexity space has fractional Helly number at most $h$ if there is a function $\beta:(0,1]\to(0,1]$ such that, for every $0<\alpha<1$ and every finite family $\mathcal F$ of convex sets, if at least an $\alpha$-fraction of the $h$-tuples of members of $\mathcal F$ have a nonempty intersection, then some point of $X$ belongs to at least $\beta(\alpha)|\mathcal F|$ members of $\mathcal F$. We call such a function $\beta$ a fractional Helly function for the space. When we want to emphasize the parameter $d=h-1$, we write this function as $\beta_d$.

We shall use the fractional Helly theorem of Holmsen and Pat\'akov\'a~\cite{HolmsenP24}: in a separable convexity space, the fractional Helly number is at most the dual VC-dimension of the corresponding halfspace system, plus one. Here the halfspace system is the set system $\mathcal H$ on the ground set $X$ whose ranges are all halfspaces of $(X,\mathcal C)$. Its dual $\mathcal H^*$ is the set system on the ground set $\mathcal H$ in which each point $x\in X$ defines the set $\{H\in\mathcal H:x\in H\}$ of halfspaces containing it. Since, in a separable convexity space with Radon number $D$, the halfspace system has VC-dimension $D-1$ (see Lemma~\ref{lem:vc-radon} below), Assouad's inequality~\cite{Assouad83}, $\operatorname{VCdim}(\mathcal H^*)<2^{\operatorname{VCdim}(\mathcal H)+1}$, shows that its dual VC-dimension is less than $2^D$. Thus, throughout the quantitative consequences below, we use $d=2^D$ and $h=d+1$ as an upper bound on the fractional Helly number.

Let $(X,\mathcal{C})$ be a convexity space, and let $P\subseteq X$ be a finite set. For $0<\varepsilon\le 1$, a set $N\subseteq X$ is called a \emph{weak $\varepsilon$-net} for $P$ if for every $C \in \mathcal{C}$ with $|C \cap P|\ge \varepsilon |P|$, one has $N\cap C\neq\emptyset$. Weak $\varepsilon$-nets for separable convexity spaces with bounded Radon number were previously obtained by Moran and Yehudayoff~\cite{MoranY20} who  showed that if a finite separable convexity space has Radon number at most $D$, then it admits weak $\varepsilon$-nets of size at most $\left(\frac{120D^2}{\varepsilon}\right)^{4D^2\ln(1/\varepsilon)}$. Thus, for fixed $D$, their bound is quasi-polynomial in $1/\varepsilon$. They also proved a converse result showing that, in compact separable convexity spaces, finite Radon number characterizes the existence of weak $\varepsilon$-nets. Holmsen and Lee~\cite{HolmsenL21} later proved a weak $\varepsilon$-net theorem in the more general setting of arbitrary, not necessarily separable, convexity spaces with bounded Radon number. Their result proceeds through a fractional Helly theorem with fractional Helly number $m=m(D)$, where their proof gives a bound on $m(D)$ which is roughly $D^{D^{\lceil\log_2 D\rceil}}$. Combined with the general theorem of Alon, Kalai, Matou\v{s}ek and Meshulam~\cite{AKMM02}, this yields weak $\varepsilon$-nets of size polynomial in $1/\varepsilon$, but with an exponent depending on this very large parameter $m(D)$.

\paragraph{The $(p,q)$-theorem.}
Let $\mathcal F\subseteq \mathcal C$ be a finite family of nonempty convex sets. We say that $\mathcal F$ satisfies the \emph{$(p,q)$-property} if among any $p$ members of $\mathcal F$, some $q$ have a nonempty intersection. A \emph{transversal} for $\mathcal F$ is a set $N\subseteq X$ such that $N\cap F\neq\emptyset$ for every $F\in\mathcal F$. We denote the minimum size of such a transversal by $\tau(\mathcal F)$.

The classical $(p,q)$-theorem of Alon and Kleitman~\cite{AlonK92} asserts that, for every $p\ge q\ge d+1$, every family of compact convex sets in $\mathbb R^d$ satisfying the $(p,q)$-property has a transversal whose size is bounded in terms of $p,q$, and $d$ only. The quantitative behavior of the transversal remains a central problem in discrete geometry. The best known general upper and lower bounds in the Euclidean setting were proved in~\cite{KellerST17} and  in~\cite{KellerS21}. In a different direction, Alon, Kalai, Matou\v{s}ek and Meshulam~\cite{AKMM02} proved a topological $(p,q)$-theorem for so-called finite good covers, without specifying quantitative bounds.

\subsection{Our results}

\textbf{Colorful $(k,r)$-shattering and a bound on the colorful VC-dimension.} The main technical contribution of this paper is a new notion of $k$-wise colorful VC-dimension and a bound on this dimension in terms of the ordinary VC-dimension. In order to present the new notion, a few more definitions are needed.


Let $H=(V,E)$ be a hypergraph. A \emph{colored set} is an ordered family $\mathcal P=(C_1,\ldots,C_\ell)$ of pairwise disjoint subsets of $V$, called color classes. Throughout this subsection, we assume that $|C_s|=r$ for all $s\in[\ell]$. Put $S=C_1\cup\cdots\cup C_\ell$.
A \emph{rainbow partition} of $S$ (w.r.t. $\mathcal P$) into $r$ parts is an ordered partition $S=R_1\dot\cup\cdots\dot\cup R_r$ such that $|R_i\cap C_s|=1$ for every $i\in[r]$ and every $s\in[\ell]$.

\begin{definition}[Colorful $(k,r)$-shattering and the colorful VC-dimension]
We say that $\mathcal P$ is \emph{colorfully $(k,r)$-shattered} by $H$ if for every rainbow partition $S=R_1\dot\cup\cdots\dot\cup R_r$ there exist distinct indices $i_1,\ldots,i_k\in[r]$ and hyperedges $e_{i_1},\ldots,e_{i_k}\in E$ such that $R_{i_j}\subseteq e_{i_j}$ for all $j=1,\ldots,k$, and $ e_{i_1}\cap\cdots\cap e_{i_k}=\emptyset$.

The \emph{colorful $(k,r)$-VC-dimension of $H$}, denoted $\operatorname{VC}^{{col}}_{k,r}(H)$, is 
the supremum of the numbers $\ell$ of color classes in a colorfully $(k,r)$-shattered colored set with color classes of size $r$. When $k,r$ are clear from the context, $\operatorname{VC}^{{col}}_{k,r}(H)$ is called `the colorful VC-dimension of $H$'.
\end{definition}

Our new definition differs from the definition of $r$-VC-dimension of Alon and Smorodinsky~\cite{AlonS26} in three aspects: the colorful structure, the $k$-wise formulation, and the strengthening of the requirement $S\cap e_{i_1}\cap\cdots\cap e_{i_k}=\emptyset$ (in~
\cite{AlonS26}) to $e_{i_1}\cap\cdots\cap e_{i_k}=\emptyset$. While the colorful aspect is, of course, needed for all the colorful theorems proved below, the direct improvement over~\cite{AlonS26} in our Tverberg theorem for separable convexity spaces, Theorem~\ref{thm:improved-tverberg} below, essentially comes only from the last two aspects mentioned above.


\medskip We prove that the colorful $(k,r)$-VC-dimension can be bounded effectively in terms of the ordinary VC-dimension.

\begin{theorem}[Colorful VC-dimension bound]\label{thm:intro-colorful-vc-bound}
Let $H=(V,E)$ be a hypergraph with $\operatorname{VCdim}(H)\le v$. Then $\operatorname{VC}^{col}_{k,r}(H)\le O(kv\log(kvr))$.
\end{theorem}

This theorem is the combinatorial core of the paper. It is a general statement about set systems of bounded VC-dimension, independent of convexity. In the geometric applications, it is applied either to the halfspace hypergraph of a separable convexity space or to auxiliary hypergraphs arising from unions of convex sets.

\paragraph{Colorful Tverberg theorems.}
We apply the colorful VC-dimension bound to separable convexity spaces. In such a space, the halfspace hypergraph has VC-dimension $D-1$, where $D$ is the Radon number (see Lemma \ref{lem:vc-radon} below). 
We start with the following colorful $k$-wise Tverberg theorem.

\begin{theorem}[Colorful Tverberg theorem -- a first tradeoff]\label{thm:colorful-tverberg}
Let $(X,\mathcal C)$ be a separable convexity space with Radon number $D$. Let $r\ge 2$ and $1\le k\le r$ be integers. There exists an absolute constant $C>0$ such that the following holds. Let $C_1,\ldots,C_\ell\subseteq X$ be pairwise disjoint color classes, each of size $r$, with $\ell\ge CkD\log(kDr)$. Then there exists a rainbow partition $C_1\cup\cdots\cup C_\ell=R_1\dot\cup\cdots\dot\cup R_r$ such that for every choice of distinct indices $i_1,\ldots,i_k\in[r]$, one has $\operatorname{conv}(R_{i_1})\cap\cdots\cap\operatorname{conv}(R_{i_k})\neq\emptyset$.
\end{theorem}

Theorem \ref{thm:colorful-tverberg}, as well as Theorem \ref{thm:large-color-classes-colored-tverberg} below, answers a question of Kalai \cite{KalaiPrivate} who asked about purely combinatorial proofs for colorful Tverberg-type theorems. It gives a rainbow partition theorem in a setting where the classical topological tools used in Euclidean colored Tverberg theorems are not available. Moreover, by Levi's theorem \cite{Levi51}, the Helly number of the space is at most $D-1$, and therefore for $k=D-1$, the $k$-wise conclusion implies that all $r$ convex hulls have a common point. Thus, Theorem~\ref{thm:colorful-tverberg} immediately yields an ordinary colored Tverberg theorem with $O(D^2\log(Dr))$ color classes, each of size $r$, in the range $r>D$. This is the first tradeoff: the color classes have the optimal size $r$, but the number of colors grows logarithmically with $r$. 

Theorem~\ref{thm:colorful-tverberg} also serves as a main tool for our colorful selection lemma. This, in turn, improves all subsequent quantitative bounds, for example, the weak $\varepsilon$-net theorem and the $(p,q)$-theorem.

Our colorful shattering technique also yields a general result for unions of convex sets. This direction is closely related to the Radon and Tverberg-type theorems for $s$-convex sets proved by Alon and Smorodinsky~\cite{AlonS26}, which address Kalai's question on Radon-type theorems for unions of convex sets. We prove a colorful extension of their Tverberg theorem. For simplicity, we state here the ordinary, non-$k$-wise, form of the result; the more general $k$-wise version is stated and proved in Section~\ref{sec:s-convex}. In a convexity space $(X,\mathcal C)$, a set $A \subset X$ is called $s$-convex if it is the union of at most $s$ convex sets.

\begin{theorem}[Colorful Tverberg theorem for $s$-convex sets]\label{thm:intro-colorful-tverberg-sconvex}
Let $(X,\mathcal C)$ be a separable convexity space with Radon number $D$, and let $s$ and $r\ge 2$ be integers. There is an absolute constant $C>0$ such that the following holds. Let $C_1,\ldots,C_\ell\subseteq X$ be color classes, each of size $r$, where
\[
\ell\ge C rD s^r\log(s^r+1)\log\!\left(r^2D s^r\log(s^r+1)\right).
\]
Then there exists a rainbow partition $C_1\cup\cdots\cup C_\ell=R_1\dot\cup\cdots\dot\cup R_r$ such that for every choice of $s$-convex sets $F_1,\ldots,F_r\subseteq X$ satisfying $R_i\subseteq F_i$ for every $i\in[r]$, one has $F_1\cap\cdots\cap F_r\neq\emptyset$.
\end{theorem}

Although this result is important in its own right, and provides a colored version of the Tverberg-type theorem for unions of convex sets, we postpone its proof to Section~\ref{sec:s-convex}. The reason is structural: the main body of the paper is organized as a chain of implications, from the colorful Tverberg theorem for separable convexity spaces to the uncolored selection lemma, the colorful selection lemma, weak $\varepsilon$-nets, and the quantitative $(p,q)$-theorem. By contrast, Theorem~\ref{thm:intro-colorful-tverberg-sconvex} is a standalone application of the same shattering technique.

The bound in Theorem~\ref{thm:intro-colorful-tverberg-sconvex} is closer in spirit to the Tverberg-type bound for unions of convex sets of \cite{AlonS26}, than to the improved bound in Theorem~\ref{thm:colorful-tverberg}. The reason is that, in the $s$-convex setting, we do not have the same Helly-type reduction that is used for convex sets. In the proof of Theorem~\ref{thm:colorful-tverberg}, Levi's theorem allows us to pass from $(D-1)$-wise intersection to full $r$-fold intersection. For unions of convex sets in a general separable convexity space, no analogous reduction is available, so in Section~\ref{sec:s-convex} we apply the $k$-wise statement with $k=r$.

\paragraph{The uncolored Tverberg consequences.}
The uncolored $k$-wise Tverberg theorem follows from Theorem~\ref{thm:colorful-tverberg} by partitioning an arbitrary point set into color classes, each of size $r$, and then forgetting the colors.

\begin{theorem}[$k$-wise Tverberg theorem]\label{thm:k-wise-tverberg}
Let $(X,\mathcal C)$ be a separable convexity space with Radon number $D$. Then every finite set $P\subseteq X$ of size at least $O(kDr\log(Dr))$ admits a partition $P=P_1\dot\cup\cdots\dot\cup P_r$ such that every $k$ of the convex hulls $\operatorname{conv}(P_1),\ldots,\operatorname{conv}(P_r)$ have a nonempty intersection.
\end{theorem}

This $k$-wise theorem is also of independent interest. In the terminology of Reay's relaxation, it gives an abstract quantitative bound on the number of points needed to force a partition whose parts are $k$-wise intersecting. As above, the ordinary Tverberg theorem follows by taking $k=D-1$: by Levi's theorem, the Helly number is at most $D-1$, and hence a $(D-1)$-wise intersecting family of convex hulls has a common point. This gives the improved Tverberg theorem.

\begin{theorem}[A Tverberg theorem for separable convexity spaces]\label{thm:improved-tverberg}
Let $(X,\mathcal C)$ be a separable abstract convexity space with Radon number $D$. Then, for every $r>D$, $\operatorname{Tv}_{\mathcal C}(r)\le O\!\left(D^2r\log r\right)$.
\end{theorem}

Thus, in separable convexity spaces, we obtain a Tverberg-type theorem that improves the Alon--Smorodinsky bound $O(Dr^2\log r)$ in the range $r>D$, and provides a step towards the weak Eckhoff conjecture in the separable setting.

\paragraph{Selection lemmas.}
We then apply a special case of the colorful Tverberg theorem (with the number of parts equal to the fractional Helly parameter $h=2^D+1$), together with the fractional Helly theorem, to obtain a colorful selection lemma (Theorem \ref{thm:colorful-selection} below). For simplicity, here we present the uncolored version.

\begin{theorem}[Selection lemma]\label{thm:selection_lemma}
Let $(X,\mathcal C)$ be a separable abstract convexity space with Radon number $D$. Then there exist an integer $a=a(D)=O(D^3)$ and a constant $\lambda>0$, depending only on $D$, such that the following holds. For every finite set $P\subseteq X$ of size $n\ge a$, there exists a point $x\in X$ that belongs to the convex hulls of at least $\lambda \binom{n}{a}$ subsets $A\subseteq P$ of size $a$. Equivalently,
$$
\bigl|\{A\in \binom{P}{a}:x\in \operatorname{conv}(A)\}\bigr|
\ge
\lambda \binom{n}{a}.
$$
\end{theorem}

In addition to being a key tool for our subsequent applications to weak $\varepsilon$-nets and quantitative $(p,q)$-theorems, this selection lemma is an abstract analogue, in the separable convexity setting, of the classical selection lemma for Euclidean convexity. The size of the selected subsets is bounded by $a=O(D^3)$, depending only polynomially on the Radon number of the space. This bound is much smaller than the superexponential bound that follows from \cite{AKMM02}. On the other hand, their result applies in a considerably more general abstract framework (in some sense).

The colorful variant of Theorem \ref{thm:selection_lemma} (Theorem \ref{thm:colorful-selection} below)
 also implies a second colored Tverberg-type statement in which the number of color classes is independent of the number of parts, and is bounded only in terms of $D$. This is closer in spirit to the classical Euclidean colored Tverberg conjecture, where the number of color classes is fixed in terms of the dimension, but here we allow each color class to be larger than $r$.

\begin{theorem}[A colorful Tverberg theorem -- a second tradeoff]\label{thm:large-color-classes-colored-tverberg}
Let $(X,\mathcal C)$ be a separable convexity space with Radon number $D$. Then there exist an integer $\ell_0=\ell_0(D)=O(D^3)$ and a constant $C_D>0$ such that the following holds. Let $r\ge 1$, and let $P_1,\ldots,P_{\ell_0}\subseteq X$ be $\ell_0$ color classes satisfying $|P_i|\ge C_Dr$ for every $i\in[\ell_0]$. Then there exist pairwise disjoint rainbow sets $A_1,\ldots,A_r$, each containing exactly one point from each color class $P_i$, such that $\bigcap_{j=1}^r\operatorname{conv}(A_j)\neq\emptyset$.
\end{theorem}

\paragraph{Weak $\varepsilon$-nets.}
Our improved selection lemma provides an improved weak $\varepsilon$-net theorem by the greedy algorithmic argument of Alon, B\'ar\'any, F\"uredi and Kleitman~\cite{AlonBFK92}.

\begin{theorem}[Weak $\varepsilon$-nets for separable convexity spaces]\label{thm:weak-epsilon-nets}
Let $(X,\mathcal C)$ be a separable convexity space with Radon number $D$. Then there exist $a=a(D)=O(D^3)$ and a constant $C_D>0$, depending only on $D$, such that the following holds. For every finite set $P\subseteq X$ and every $0<\varepsilon\le 1$, there exists a weak $\varepsilon$-net $N\subseteq X$ for $P$ with $|N|\le C_D\varepsilon^{-a}$.
\end{theorem}

 Compared with Moran and Yehudayoff~\cite{MoranY20}, the dependence on $1/\varepsilon$ is improved from quasi-polynomial to polynomial for fixed $D$. The bound of~\cite{MoranY20}, however, holds in a more general setting where one only requires point-convex set separation. Compared with the general result of Holmsen and Lee~\cite{HolmsenL21}, which applies without separability, the exponent obtained here in the separable setting is substantially smaller\footnote{Their result proceeds through a fractional Helly theorem with fractional Helly number $m=m(D)$, where the available bound on $m(D)$ is roughly $D^{D\lceil\log_2D\rceil}$. }. 

\paragraph{A quantitative $(p,q)$-theorem.}
Let $(X,\mathcal C)$ be a separable convexity space with Radon number $D$, put $d=2^D$, and set $h=d+1$. Let $\beta_d(\alpha)$ denote a fractional Helly function for the space with fractional Helly number at most $h$.

\begin{theorem}[A quantitative \texorpdfstring{$(p,q)$}{(p,q)}-theorem]\label{thm:pq-separable}
Let $(X,\mathcal C)$ be a separable convexity space with Radon number $D$, and put $d=2^D$ and $h=d+1$. Let $a=a(D)=O(D^3)$ and $C_D$ be as in Theorem~\ref{thm:weak-epsilon-nets}. Then for every $p\ge q\ge h$, every finite family $\mathcal F\subseteq\mathcal C$ of nonempty convex sets satisfying the $(p,q)$-property admits a transversal of size at most
$$
C_D\left(\beta_d\left(\frac{\binom{q}{h}}{\binom{(p-1)(q-1)+1}{h}}\right)\right)^{-a}.
$$
\end{theorem}
The theorem is stated in terms of the fractional Helly function $\beta_d$, since in this level of generality the dependence of $\beta_d(\alpha)$ on $\alpha$ is part of the quantitative input.

We end the paper with a brief discussion of what happens when our separation requirement is dropped to require only point--convex-set separation. Under a mild compactness assumption for finitely generated convex sets, the same shattering method still applies, after replacing halfspaces by intersections of $O(D)$ halfspaces. This yields analogues of the main results with slightly weaker dependence on $D$; see Section~\ref{sec:weak-separability}.

The logical flow of the main quantitative results is summarized in the following diagram.

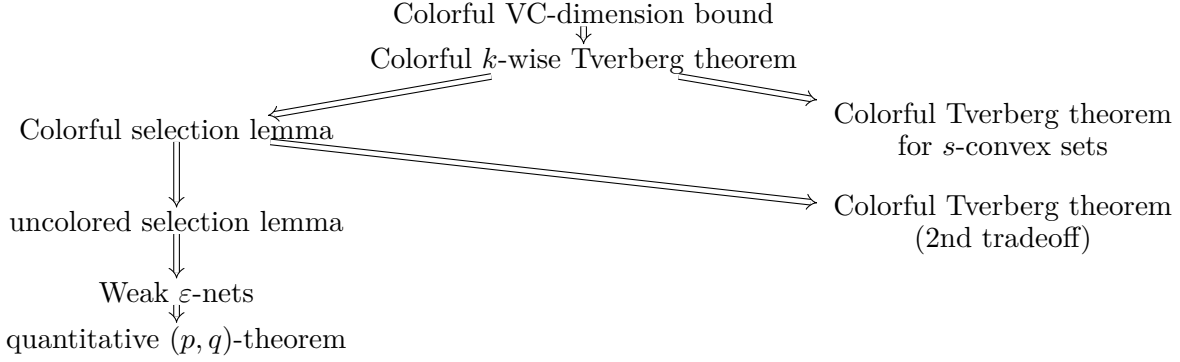
\[
\begin{tikzcd}[
  column sep=0.6em,
  row sep=0.65em,
  cells={nodes={inner xsep=1pt, inner ysep=1pt}}
]
& \text{Colorful VC-dimension bound} \arrow[d, Rightarrow] & \\
& \text{Colorful $k$-wise Tverberg theorem}
  \arrow[dl, Rightarrow]
  \arrow[dr, Rightarrow]
& \\
\text{Colorful selection lemma}
  \arrow[d, Rightarrow]
  \arrow[drr, Rightarrow]
& &
\begin{array}{c}
\text{Colorful Tverberg theorem}\\[-0.15em]
\text{for $s$-convex sets}
\end{array}
\\
\text{uncolored selection lemma} \arrow[d, Rightarrow]
& &
\begin{array}{c}
\text{Colorful Tverberg theorem}\\[-0.15em]
\text{(2nd tradeoff)}
\end{array}
\\
\text{Weak $\varepsilon$-nets} \arrow[d, Rightarrow] & & \\
\text{quantitative $(p,q)$-theorem} & &
\end{tikzcd}
\]

\paragraph{Organization of the paper.}
In Section~\ref{sec:colorful-vc}  we introduce the colorful rainbow shattering notion and the colorful VC-dimension bound. In Section~\ref{sec:convexity-tverberg} we apply this bound to the halfspace hypergraph of a separable convexity space and derive the colorful and uncolored Tverberg theorems. In Section~\ref{sec:selection} we prove the colorful selection lemma, the uncolored selection lemma, and the second colored Tverberg tradeoff. In Sections~\ref{sec:epsnet} and~\ref{sec:pq} we derive the weak $\varepsilon$-net theorem and the quantitative $(p,q)$-theorem. In Section~\ref{sec:s-convex} we prove the colorful Tverberg theorem for unions of convex sets, and finally, in Section~\ref{sec:weak-separability} we discuss the variant under the weaker point--convex-set separation assumption.

\section{Colorful VC-dimension and rainbow shattering}\label{sec:colorful-vc}

We begin with the purely combinatorial part of the paper. Let $H=(V,E)$ be a hypergraph.
Alon and Smorodinsky~\cite{AlonS26} introduced an $r$-shattering notion suited to Tverberg-type problems. In one form, a finite set $S$ is $r$-shattered if for every ordered partition $S=S_1\dot\cup\cdots\dot\cup S_r$, one can find hyperedges $e_1,\ldots,e_r$ such that $S_i\subseteq e_i$ for every $i$ and $S \cap \left(\bigcap_{i=1}^r e_i \right)=\emptyset$. The variants below refine this idea in two directions: they allow only $k$ labels to be used in the certificate, and they impose a rainbow structure on the partitions. These two refinements are what makes the later $k$-wise and colorful Tverberg theorems possible.

\subsection{Colorful and rainbow shattering}

Recall that a \emph{colored set} in a hypergraph $H=(V,E)$, is an ordered family $\mathcal P=(C_1,\ldots,C_\ell)$ of pairwise disjoint subsets of $V$, called color classes. Throughout this subsection we assume that $|C_s|=r$ for all $s\in[\ell]$. Put $S=C_1\cup\cdots\cup C_\ell$.
A \emph{rainbow partition} of $S$ (w.r.t.~$\mathcal P$) into $r$ parts is an ordered partition $S=R_1\dot\cup\cdots\dot\cup R_r$ such that $|R_i\cap C_s|=1$ for every $i\in[r]$ and every $s\in[\ell]$.
$\mathcal P$ is \emph{colorfully $(k,r)$-shattered} by $H$ if for every rainbow partition $S=R_1\dot\cup\cdots\dot\cup R_r$ there exist distinct indices $i_1,\ldots,i_k\in[r]$ and hyperedges $e_{i_1},\ldots,e_{i_k}\in E$ such that $R_{i_j}\subseteq e_{i_j}$ for all $j=1,\ldots,k$, and $ e_{i_1}\cap\cdots\cap e_{i_k}=\emptyset$.
$\operatorname{VC}^{col}_{k,r}(H)$ denotes the \emph{colorful VC-dimension of $H$}, namely, the supremum of the numbers $\ell$ of color classes in a colorfully $(k,r)$-shattered colored set with color classes of size $r$.

\medskip

\noindent \textbf{Theorem \ref{thm:intro-colorful-vc-bound} -- restatement} (Colorful VC-dimension bound).
Let $H=(V,E)$ be a hypergraph with $\operatorname{VCdim}(H)\le v$, where $v\ge 1$. Then, for all $1\le k\le r$ with $r\ge 2$, $\operatorname{VC}^{col}_{k,r}(H)\le O(kv\log(kvr))$.

\medskip

\begin{proof}
Let $\mathcal P=(C_1,\ldots,C_\ell)$ be colorfully $(k,r)$-shattered by $H$, and put $S=C_1\cup\cdots\cup C_\ell$. Thus $|S|=\ell r$.

The number of rainbow partitions of $S$ into $r$ ordered parts is $(r!)^\ell$, since for each color class one chooses a bijection from its $r$ points to the $r$ parts.

Every rainbow partition is certified by an ordered $k$-tuple of distinct labels $i_1,\ldots,i_k\in[r]$ and by an ordered $k$-tuple of traces $A_j=e_{i_j}\cap S$, $j=1,\ldots,k$, such that $R_{i_j}\subseteq A_j$ for all $j$, and $A_1\cap\cdots\cap A_k=\emptyset$. The number of choices of the labels is at most $r^k$. By the Perles--Sauer--Shelah lemma, the number of traces of edges of $H$ on $S$ is at most $C(\ell r)^v$, for an absolute constant $C$. Hence the number of possible ordered $k$-tuples of traces is at most $(C(\ell r)^v)^k$.

Next, fix the labels $i_1,\ldots,i_k$ and the traces $A_1,\ldots,A_k$. We bound the number of rainbow partitions certified by this fixed choice. Consider a single color class $C_s$. We form an $r\times r$ matrix $M_s$ whose rows are indexed by the points of $C_s$ and whose columns are indexed by the parts $1,\ldots,r$. We put $(M_s)_{x,t}=1$ if the point $x$ is allowed to be placed in the part $R_t$ under the fixed certificate, and put $(M_s)_{x,t}=0$ otherwise. Thus, if $t=i_j$ for some $j$, then $x$ is allowed in the part $R_t$ only if $x\in A_j$.

Since $A_1\cap\cdots\cap A_k=\emptyset$, every point $x\in C_s$ is missing from at least one of the traces $A_j$. Therefore, each row of $M_s$ contains at most $r-1$ ones. The number of admissible bijective assignments of the points of $C_s$ to the parts is exactly the permanent $\operatorname{per}(M_s)$. If some row is zero, this number is zero and the desired bound is trivial. Otherwise, if the row sums are $b_1,\ldots,b_r$, then $1\le b_i\le r-1$ for all $i$. By the Bregman--Minc inequality \cite{Bregman73},
\[
\operatorname{per}(M_s)\le \prod_{i=1}^r ((b_i!)^{1/b_i})\le \left((r-1)!\right)^{r/(r-1)}.
\]
The right inequality holds since the function $b\mapsto (b!)^{1/b}$ is increasing for $b\ge 1$. Indeed, $(b!)^{1/b}\le ((b+1)!)^{1/(b+1)}$ is equivalent\footnote{If $t=b!$ then both inequalities are equivalent to the inequality $t^{b+1}\leq (b+1)^b\cdot t^b$.} to $b!\le (b+1)^b$, and the latter holds by the geometric--arithmetic mean inequality on the numbers $1,\ldots,b$. Thus, for the fixed certificate, the number of rainbow partitions it can certify is at most $\left(\left((r-1)!\right)^{r/(r-1)}\right)^\ell$.

Combining the estimates gives
\[
(r!)^\ell
\le
r^k\bigl(C(\ell r)^v\bigr)^k
\left(\left((r-1)!\right)^{r/(r-1)}\right)^\ell.
\]
Equivalently,
\[
\left(
\frac{r!}{\left((r-1)!\right)^{r/(r-1)}}
\right)^\ell
\le
r^k\bigl(C(\ell r)^v\bigr)^k.
\]
Let $\rho_r=\frac{r!}{\left((r-1)!\right)^{r/(r-1)}}=\frac{r}{\left((r-1)!\right)^{1/(r-1)}}$. By the arithmetic-geometric mean inequality applied to $1,2,\ldots,r-1$, we have $\left((r-1)!\right)^{1/(r-1)}\le r/2$. Hence $\rho_r\ge 2$ for all $r\ge 2$. Taking logarithms, we obtain $\ell\le O(k\log r+kv\log(\ell r))\le O(kv\log(\ell r))$.

The standard estimate $m\le A\log(Bm)$ implies $m=O(A\log(AB))$. Applying it with $m=\ell$, $A=O(kv)$ and $B=r$, we get $\ell\le O(kv\log(kvr))$.
\end{proof}

\section{From colorful VC-dimension to Tverberg theorems}\label{sec:convexity-tverberg}

We now pass from the general colorful VC framework to separable convexity spaces. 

\subsection{The halfspace hypergraph}

Let $(X,\mathcal C)$ be a separable convexity space. Let $H_{\mathcal C}$ be the halfspace hypergraph of the space: its vertex set is $X$, and its hyperedges are the halfspaces of $(X,\mathcal C)$.

We shall use the following simple observation, noted by Alon and Smorodinsky~\cite{AlonS26}.

\begin{lemma}\label{lem:vc-radon}
Let $(X,\mathcal C)$ be a separable convexity space with Radon number $D$. Then the halfspace hypergraph $H_{\mathcal C}$ has VC-dimension $\operatorname{VCdim}(H_{\mathcal C})=D-1$.
\end{lemma}

\begin{proof}
First, suppose that $S\subseteq X$ is shattered by halfspaces. We claim that $S$ has no Radon partition. Indeed, for every partition $S=A\dot\cup B$, since $S$ is shattered, there is a halfspace $h$ such that $h\cap S=A$. Hence $\operatorname{conv}(A)\subseteq h$ and $\operatorname{conv}(B)\subseteq X\setminus h$, so $\operatorname{conv}(A)\cap \operatorname{conv}(B)=\emptyset$. Thus no partition of $S$ is a Radon partition. By the definition of the Radon number, this implies $|S|\le D-1$. Therefore $\operatorname{VCdim}(H_{\mathcal C})\le D-1$.

Conversely, by the minimality of the Radon number $D$, there exists a set $S\subseteq X$ of size $D-1$ with no Radon partition. We show that $S$ is shattered by halfspaces. Let $A\subseteq S$. Since $S$ has no Radon partition, the convex sets $\operatorname{conv}(A)$ and $\operatorname{conv}(S\setminus A)$ are disjoint. By separability, there exists a halfspace $h$ such that $\operatorname{conv}(A)\subseteq h$ and $h \cap \operatorname{conv}(S\setminus A)= \emptyset$.  Thus every subset $A\subseteq S$ is realized as the trace of a halfspace on $S$. Hence, $S$ is shattered, and so $\operatorname{VCdim}(H_{\mathcal C})\ge D-1$.
\end{proof}

\subsection{From colorful shattering to colorful Tverberg partitions}

Let $(X,\mathcal C)$ be a separable convexity space, and let $H_{\mathcal C}$ be its halfspace hypergraph. For $1\le k\le r$, define $\operatorname{CT}_{\mathcal C}(k,r)$ to be the smallest integer $\ell$ such that for every colored set $\mathcal P=(C_1,\ldots,C_\ell)$ with $|C_s|=r$ for all $s$, there exists a rainbow partition $C_1\cup\cdots\cup C_\ell=R_1\dot\cup\cdots\dot\cup R_r$ such that every $k$ of the convex hulls $\operatorname{conv}(R_1),\ldots,\operatorname{conv}(R_r)$ have a nonempty intersection.

\medskip

\noindent \textbf{Theorem \ref{thm:colorful-tverberg} -- a formal restatement} (Colorful $k$-wise Tverberg theorem).
Let $(X,\mathcal C)$ be a separable convexity space with Radon number $D$. Then $\operatorname{CT}_{\mathcal C}(k,r)\le O(kD\log(kDr))$. 

\begin{proof}
By Lemma~\ref{lem:vc-radon}, the halfspace hypergraph $H_{\mathcal C}$ has VC-dimension $D-1$. By Theorem~\ref{thm:intro-colorful-vc-bound}, no colored set with more than $O(kD\log(kDr))$ color classes can be colorfully $(k,r)$-shattered by $H_{\mathcal C}$.

Let $\mathcal P=(C_1,\ldots,C_\ell)$ be a colored set with $\ell$ larger than this bound. Since $\mathcal P$ is not colorfully $(k,r)$-shattered, there exists a rainbow partition $R_1,\ldots,R_r$ such that for every choice of distinct indices $i_1,\ldots,i_k\in[r]$ and every choice of halfspaces $h_{i_1},\ldots,h_{i_k}$ with $R_{i_j}\subseteq h_{i_j}$ for all $j$, we have $h_{i_1}\cap\cdots\cap h_{i_k}\neq\emptyset$.

We claim that every $k$ of the convex hulls $\operatorname{conv}(R_1),\ldots,\operatorname{conv}(R_r)$ intersect. Otherwise, for some distinct $i_1,\ldots,i_k$, the convex sets $K_j=\operatorname{conv}(R_{i_j})$, $j=1,\ldots,k$, have an empty intersection.
We now replace the convex sets $K_1,\ldots,K_k$ by halfspaces, one at a time, while preserving the empty-intersection property. First consider $K_1$. The set $B_1=K_2\cap\cdots\cap K_k$ is convex and disjoint from $K_1$. By the first separation property, there is a halfspace $h_{i_1}$ such that $K_1\subseteq h_{i_1}$ and $B_1\subseteq X\setminus h_{i_1}$. Hence $h_{i_1}\cap K_2\cap\cdots\cap K_k=\emptyset$, and of course $R_{i_1}\subseteq h_{i_1}$.

Next, we replace $K_2$. Now the other side is the convex set $B_2=h_{i_1}\cap K_3\cap\cdots\cap K_k$, which is disjoint from $K_2$ because $h_{i_1}\cap K_2\cap\cdots\cap K_k=\emptyset$. Applying the same separation property, we find a halfspace $h_{i_2}$ such that $K_2\subseteq h_{i_2}$ and $B_2\subseteq X\setminus h_{i_2}$. Thus, $h_{i_1}\cap h_{i_2}\cap K_3\cap\cdots\cap K_k=\emptyset$, while $R_{i_2}\subseteq h_{i_2}$.
Continuing in this way, replacing one remaining convex set by a separating halfspace at each step, we obtain halfspaces $h_{i_1},\ldots,h_{i_k}$ such that $R_{i_j}\subseteq h_{i_j}$ for all $j=1,\ldots,k$ and $h_{i_1}\cap\cdots\cap h_{i_k}=\emptyset$. This contradicts the choice of the rainbow partition. Hence, the partition has the desired $k$-wise intersection property.
\end{proof}

We now derive the usual, non-$k$-wise, colorful Tverberg consequence.

\begin{corollary}[Colorful Tverberg theorem, first tradeoff]\label{cor:colorful-full-tverberg}
Let $(X,\mathcal C)$ be a separable convexity space with Radon number $D$. For every $r>D$, $O(D^2\log(Dr))$ color classes, each of size $r$, are sufficient to force a rainbow partition into $r$ parts whose convex hulls have a common point.
\end{corollary}

\begin{proof}
Apply Theorem~\ref{thm:colorful-tverberg} with $k=D-1$. This gives a rainbow partition whose convex hulls are $(D-1)$-wise intersecting. By Levi's theorem~\cite{Levi51}, the Helly number of $(X,\mathcal C)$ is at most $D-1$. Therefore all $r$ convex hulls have a nonempty intersection.
\end{proof}

For the results in Section \ref{sec:selection}, we shall also use the colorful theorem with the number of parts equal to an upper bound on the fractional Helly number. Put $d=2^D$ and $h=d+1$. Recall that by \cite{HolmsenP24}, $h$ is an upper bound on the fractional Helly number of $(X,\mathcal C)$.

\begin{corollary}[A colorful restricted Tverberg theorem]\label{cor:colorful-restricted-tverberg}
Let $(X,\mathcal C)$ be a separable convexity space with Radon number $D$. Put $d=2^D$ and $h=d+1$. There exists an integer $\ell=\ell(D)=O(D^3)$ such that the following holds. For every colored set $\mathcal P=(C_1,\ldots,C_\ell)$ with $|C_s|=h$ for all $s$, there exists a rainbow partition $C_1\cup\cdots\cup C_\ell=R_1\dot\cup\cdots\dot\cup R_h$ such that $\bigcap_{i=1}^h \operatorname{conv}(R_i)\neq\emptyset$.
\end{corollary}

\begin{proof}
Apply Theorem~\ref{thm:colorful-tverberg} with $r=h=d+1$ and $k=D-1$. Since $d=2^D$, we have
\[
O(kD\log(kDh))=O(D^2\log(D^2 2^D))=O(D^3).
\]
Thus $\ell=O(D^3)$ color classes suffice to obtain a rainbow partition whose convex hulls are $(D-1)$-wise intersecting. By Levi's theorem~\cite{Levi51}, the Helly number of $(X,\mathcal C)$ is at most $D-1$. Therefore the whole family of $h$ convex hulls has nonempty intersection.
\end{proof}

\subsection{The uncolored consequences}

We now deduce the uncolored $k$-wise Tverberg theorem from the colorful theorem by partitioning an arbitrary point set into color classes, each of size $r$, and then forgetting the colors.

\medskip

\noindent \textbf{Theorem~\ref{thm:k-wise-tverberg} -- restatement} ($k$-wise Tverberg Theorem).
Let $(X,\mathcal C)$ be a separable convexity space with Radon number $D$. Then every finite set $P\subseteq X$ of size at least $O(kDr\log(Dr))$ admits a partition $P=P_1\dot\cup\cdots\dot\cup P_r$ such that every $k$ of the convex hulls $\operatorname{conv}(P_1),\ldots,\operatorname{conv}(P_r)$ have nonempty intersection.

\medskip

\begin{proof}
Let $\ell=O(kD\log(kDr))$ be large enough for Theorem~\ref{thm:colorful-tverberg}. Suppose $|P|\ge \ell r$. Choose pairwise disjoint subsets $C_1,\ldots,C_\ell\subseteq P$, each of size $r$, and let $S=C_1\cup\cdots\cup C_\ell$. Applying Theorem~\ref{thm:colorful-tverberg} to the colored set $(C_1,\ldots,C_\ell)$, we obtain a rainbow partition $S=R_1\dot\cup\cdots\dot\cup R_r$ such that every $k$ of the convex hulls $\operatorname{conv}(R_1),\ldots,\operatorname{conv}(R_r)$ have nonempty intersection.
Now add the remaining points of $P\setminus S$ arbitrarily to the parts $R_1,\ldots,R_r$, obtaining the desired partition $P=P_1\dot\cup\cdots\dot\cup P_r$ with $R_i\subseteq P_i$ for every $i$. Since $\ell r=O(kDr\log(kDr))$ and $k\le r$ implies $\log(kDr)=O(\log(Dr))$, this proves the theorem.
\end{proof}

We finally deduce the improved Tverberg bound.

\medskip

\noindent \textbf{Theorem~\ref{thm:improved-tverberg} -- restatement} (The Tverberg Theorem).
Let $(X,\mathcal C)$ be a separable convexity space with Radon number $D$. Then, for every $r>D$, $\operatorname{Tv}_{\mathcal C}(r)\le O(D^2r\log r)$.

\medskip

\begin{proof}
Apply Theorem~\ref{thm:k-wise-tverberg} with $k=D-1$. 
By Levi's theorem~\cite{Levi51}, the Helly number of $(X,\mathcal C)$ is at most $D-1$. Therefore the whole family $\operatorname{conv}(P_1),\ldots,\operatorname{conv}(P_r)$ has a nonempty intersection. 
\end{proof}

\section{Colorful and uncolored selection lemmas}\label{sec:selection}

In this section we use the colorful restricted Tverberg theorem to prove a colorful selection lemma with $O(D^3)$ colors. We then derive the improved uncolored selection lemma. Interestingly, working directly with the uncolored result would imply much weaker bounds in the selection lemma that we obtain.

Throughout this section, let $d=2^D$ and $h=d+1$, and let $\ell=\ell(D)=O(D^3)$ be the integer from Corollary~\ref{cor:colorful-restricted-tverberg}. Let $\beta_d$ denote a fractional Helly function for $(X,\mathcal C)$ with fractional Helly number at most $h=d+1$.

\subsection{The selection lemmas}

\begin{theorem}[Colorful selection lemma]\label{thm:colorful-selection}
Let $(X,\mathcal C)$ be a separable convexity space with Radon number $D$. Then there exist an integer $\ell=\ell(D)=O(D^3)$ and a constant $\lambda_{\mathrm{col}}>0$, depending only on $D$, such that the following holds. Let $P_1,\ldots,P_\ell\subseteq X$ be pairwise disjoint finite sets with $|P_1|=\cdots=|P_\ell|=n\ge h$. Then there exists a point $x\in X$ such that
\[
\bigl|\{(p_1,\ldots,p_\ell)\in P_1\times\cdots\times P_\ell:
 x\in \operatorname{conv}\{p_1,\ldots,p_\ell\}\}\bigr|
\ge
\lambda_{\mathrm{col}} n^\ell.
\]
\end{theorem}

\begin{proof}
Consider the finite family, with multiplicities,
\[
\mathcal F=
\left\{
\operatorname{conv}\{p_1,\ldots,p_\ell\}:p_i\in P_i\text{ for all }i\in[\ell]
\right\}.
\]
Here each colorful choice $(p_1,\ldots,p_\ell)$ labels a separate copy of the corresponding convex set. Thus $\mathcal F$ is a multiset of cardinality $n^\ell$, and its $h$-tuples are $h$-element subfamilies of these labelled copies.

We show that a positive fraction, depending only on $D$, of all $h$-element subfamilies of labelled copies in $\mathcal F$ have a nonempty intersection. Choose subsets $Q_i\subseteq P_i$ with $|Q_i|=h$ for all $i=1,\ldots,\ell$. There are $\binom{n}{h}^\ell$ choices. For each such choice, Corollary~\ref{cor:colorful-restricted-tverberg} applied to the color classes $Q_1,\ldots,Q_\ell$ gives $h$ pairwise disjoint colorful sets
$A_1,\ldots,A_h$, each containing one point from each $Q_i$, such that
$\bigcap_{j=1}^h \operatorname{conv}(A_j)\neq\emptyset.$
Therefore each choice of $Q_1,\ldots,Q_\ell$ produces an intersecting $h$-tuple of members of $\mathcal F$.

Conversely, any $h$ pairwise disjoint colorful sets determine the corresponding sets $Q_i$ uniquely, namely by taking the $h$ points of color $i$ which appear in them. Hence the number of intersecting $h$-tuples in $\mathcal F$ is at least $\binom{n}{h}^\ell$. Consequently, the fraction of intersecting $h$-tuples is at least
\[
\alpha_n=\frac{\binom{n}{h}^\ell}{\binom{n^\ell}{h}}.
\]
Using $\binom{n}{h}\ge (n/h)^h$ and $\binom{n^\ell}{h}\le (e n^\ell/h)^h$, we get the uniform lower bound
\[
\alpha_n\ge e^{-h}h^{-h(\ell-1)}=:\alpha_{\mathrm{col}}>0.
\]

Since $h=d+1$ is a fractional Helly number for the space, the fractional Helly theorem applied to $\mathcal F$ gives a point $x\in X$ contained in at least
$\beta_d(\alpha_{\mathrm{col}})|\mathcal F|
=
\beta_d(\alpha_{\mathrm{col}}) n^\ell$
members of $\mathcal F$. Thus the theorem holds with
$\lambda_{\mathrm{col}}=\beta_d(\alpha_{\mathrm{col}}).$
\end{proof}

Here we derive:

\medskip

\noindent \textbf{Theorem~\ref{thm:selection_lemma} -- restatement} (A selection lemma).
Let $(X,\mathcal C)$ be a separable abstract convexity space with Radon number $D$. Then there exist an integer $a=a(D)=O(D^3)$ and a constant $\lambda>0$, depending only on $D$, such that the following holds. For every finite set $P\subseteq X$ of size $n\ge a$, there exists a point $x\in X$ that belongs to the convex hulls of at least $\lambda\binom{n}{a}$ subsets $A\subseteq P$ of size $a$. 

\begin{proof}
Let $\ell=\ell(D)$ be the number of colors from Theorem~\ref{thm:colorful-selection}. We first assume that $n\ge ah$. Put $s=\lfloor n/a\rfloor$. Then $s\ge h$, and after discarding at most $a-1$ points of $P$ we may choose pairwise disjoint subsets $P_1,\ldots,P_a\subseteq P$ with
$|P_1|=\cdots=|P_a|=s$.

Applying Theorem~\ref{thm:colorful-selection} to $P_1,\ldots,P_a$, we find a point $x\in X$ that belongs to the convex hulls of at least $\lambda_{\mathrm{col}}s^a$ colorful $a$-tuples. Since the color classes are disjoint, each such colorful tuple is an $a$-element subset of $P$. Thus
\[
\bigl|\{A\in\binom{P}{a}:x\in\operatorname{conv}(A)\}\bigr|
\ge
\lambda_{\mathrm{col}}s^a.
\]
As $n\ge ah$ and $h\ge 2$, we have $s=\lfloor n/a\rfloor\ge n/(2a)$. Hence
\[
\lambda_{\mathrm{col}}s^a
\ge
\lambda_{\mathrm{col}}\left(\frac{n}{2a}\right)^a
\ge
\lambda_{\mathrm{col}}\frac{a!}{(2a)^a}\binom{n}{a}.
\]

For the remaining finitely many values $a\le n<ah$, we decrease the constant if necessary.
\end{proof}

\subsection{A colored Tverberg consequence with fewer colors}

We record one further consequence of the colorful selection lemma. It gives a colored Tverberg-type theorem in which the number of color classes depends only polynomially on the Radon number, while the size of each color class is allowed to be larger than the number of desired parts.
The derivation of a colored Tverberg-type statement from a selection lemma is a standard technique. Alon, B\'ar\'any, F\"uredi and Kleitman~\cite{AlonBFK92} observed in the Euclidean setting that point selection, hitting-set, and multicolored Tverberg theorems are closely related, indeed equivalent up to changes in the parameters. 

\medskip

\noindent \textbf{Theorem~\ref{thm:large-color-classes-colored-tverberg} -- restatement} (A colorful Tverberg Theorem -- second tradeoff).
Let $(X,\mathcal C)$ be a separable convexity space with Radon number $D$. Then there exist an integer $\ell_0=\ell_0(D)=O(D^3)$ and a constant $C_D>0$ such that the following holds. Let $r\ge 1$, and let $P_1,\ldots,P_{\ell_0}\subseteq X$ be $\ell_0$ color classes satisfying $|P_i|\ge C_Dr$ for every $i\in[\ell_0]$. Then there exist pairwise disjoint rainbow sets $A_1,\ldots,A_r$, each containing exactly one point from each color class $P_i$, such that $\bigcap_{j=1}^r \operatorname{conv}(A_j)\neq\emptyset$.

\begin{proof}
Let $\ell_0=\ell_0(D)=O(D^3)$ and $\lambda_{\mathrm{col}}>0$ be the constants given by the colorful selection lemma. Thus, for any $\ell_0$ color classes of the same size $M$, there exists a point $x\in X$ that belongs to the convex hulls of at least $\lambda_{\mathrm{col}}M^{\ell_0}$ colorful choices.

Choose
\(
C_D\ge \max\left\{h,\frac{2\ell_0}{\lambda_{\mathrm{col}}}+1\right\}.
\)
Let $P_1,\ldots,P_{\ell_0}$ be color classes with $|P_i|\ge C_Dr$. Put
\(  
M=\max\left\{h,\left\lceil \frac{2\ell_0r}{\lambda_{\mathrm{col}}}\right\rceil\right\}.
\)
Then $M\le C_Dr\le |P_i|$ for every $i\in[\ell_0]$.

Applying the colorful selection lemma to $P_1',\ldots,P_{\ell_0}'$, we obtain a point $x\in X$ such that $x\in \operatorname{conv}{p_1,\ldots,p_{\ell_0}}$ for at least $\lambda_{\mathrm{col}}M^{\ell_0}$ choices $(p_1,\ldots,p_{\ell_0})\in P_1'\times\cdots\times P_{\ell_0}'$.

Define an $\ell_0$-partite $\ell_0$-uniform hypergraph $\mathcal H_x$ with vertex classes $P_1',\ldots,P_{\ell_0}'$. Its edges are the colorful sets ${p_1,\ldots,p_{\ell_0}}$, where $p_i\in P_i'$, such that $x\in \operatorname{conv}\{p_1,\ldots,p_{\ell_0}\}$. By the choice of $x$, this hypergraph has at least $\lambda_{\mathrm{col}}M^{\ell_0}$ edges.

We claim that $\mathcal H_x$ contains a matching of size $r$. Suppose not. Let $\mathcal M$ be a maximal matching in $\mathcal H_x$. Then $|\mathcal M|\le r-1$, and the union of the edges of $\mathcal M$ contains at most $\ell_0(r-1)$ vertices. By maximality, every edge of $\mathcal H_x$ must meet this union; otherwise it could be added to $\mathcal M$. On the other hand, each vertex is contained in at most $M^{\ell_0-1}$ colorful edges. Therefore $|E(\mathcal H_x)|\le \ell_0(r-1)M^{\ell_0-1}$. This contradicts the lower bound $|E(\mathcal H_x)|\ge \lambda_{\mathrm{col}}M^{\ell_0}$, because our choice of $M$ gives $\lambda_{\mathrm{col}}M>\ell_0(r-1)$. Hence $\mathcal H_x$ contains a matching of size $r$.

Let $A_1,\ldots,A_r$ be the edges of such a matching. They are pairwise disjoint rainbow sets, each containing one point from each color class. Moreover, by the definition of $\mathcal H_x$, we have $x\in\operatorname{conv}(A_j)$ for every $j=1,\ldots,r$. Hence $x\in \bigcap_{j=1}^r \operatorname{conv}(A_j)$, as required.
\end{proof}

\section{Weak \texorpdfstring{$\varepsilon$}{epsilon}-nets}\label{sec:epsnet}

In this section we derive the improved weak $\varepsilon$-net theorem from the improved selection lemma. The argument is the standard greedy argument, as in~\cite{AlonBFK92}.

\medskip
\noindent \textbf{Theorem~\ref{thm:weak-epsilon-nets} -- restatement.}
Let $(X,\mathcal C)$ be a separable convexity space with Radon number $D$. Then there exist $a=a(D)=O(D^3)$ and a constant $C_D>0$, depending only on $D$, such that the following holds. For every finite set $P\subseteq X$ and every $0<\varepsilon\le 1$, there exists a weak $\varepsilon$-net $N\subseteq X$ for $P$ with $|N|\le C_D\varepsilon^{-a}$.

\medskip

\begin{proof}
Let $a=a(D)=O(D^3)$ and $\lambda>0$ be the constants from Theorem~\ref{thm:selection_lemma}. Let $P\subseteq X$ be finite, and write $n=|P|$.
If $\varepsilon n<a$, then $N=P$ is a weak $\varepsilon$-net and
$|N|=n<a\varepsilon^{-1}\le a\varepsilon^{-a}$. Thus we may assume that $s:=\lceil\varepsilon n\rceil\ge a$.

We construct $N$ greedily. At any stage, call a subset $A\in\binom{P}{a}$ \emph{alive} if $N\cap\operatorname{conv}(A)=\emptyset$. Initially all $a$-subsets are alive. Suppose that the current set $N$ is not a weak $\varepsilon$-net for $P$. Then there exists $Q\subseteq P$ with $|Q|\ge s$ such that $N\cap\operatorname{conv}(Q)=\emptyset$. By the selection lemma applied to $Q$, there is a point $x\in X$ such that $x\in\operatorname{conv}(A)$ for at least $\lambda\binom{|Q|}{a}\ge\lambda\binom{s}{a}$ subsets $A\in\binom{Q}{a}$.

All these $a$-subsets are alive before $x$ is added to $N$, because $\operatorname{conv}(A)\subseteq\operatorname{conv}(Q)$ and $N\cap\operatorname{conv}(Q)=\emptyset$. After adding $x$, they are no longer alive. Thus each step kills at least $\lambda\binom{s}{a}$ alive $a$-subsets. Since the initial number of alive $a$-subsets is $\binom{n}{a}$, the number of greedy steps is at most
$\lambda^{-1}\frac{\binom{n}{a}}{\binom{s}{a}}.$
When the procedure stops, $N$ is a weak $\varepsilon$-net.

Using again $\binom{n}{a}\le (en/a)^a$, $\binom{s}{a}\ge (s/a)^a$, and $s\ge\varepsilon n$, we obtain
\[
\frac{\binom{n}{a}}{\binom{s}{a}}
\le
\left(\frac{en}{s}\right)^a
\le
\left(\frac{e}{\varepsilon}\right)^a.
\]
Hence $|N|\le \lambda^{-1}e^a\varepsilon^{-a}$. Combining this with the trivial case above gives the result with $C_D=\max\{a,\lambda^{-1}e^a\}$.
\end{proof}

\begin{remark}
Theorem~\ref{thm:weak-epsilon-nets} gives weak $\varepsilon$-nets of size $O_D(\varepsilon^{-a})$, now with $a=O(D^3)$. The constant hidden in $O_D(\cdot)$ still depends on the fractional Helly function through the constant $\lambda$ in the selection lemma.
\end{remark}

\section{A quantitative \texorpdfstring{$(p,q)$}{(p,q)}-theorem}\label{sec:pq}

We now derive the improved quantitative $(p,q)$-theorem. The proof follows the classical Alon--Kleitman \cite{AlonK92} scheme: first one finds a positive fraction of intersecting subfamilies, then applies fractional Helly, then uses linear programming duality, and finally applies the weak $\varepsilon$-net theorem.

Let $(X,\mathcal C)$ be a separable convexity space with Radon number $D$. As above, put $d=2^D$ and $h=d+1$, so that $h$ is an upper bound on the fractional Helly number of the space (as was proved in \cite{HolmsenP24}). Let $\beta_d(\alpha)$ denote a corresponding fractional Helly function.

\medskip
\noindent \textbf{Theorem~\ref{thm:pq-separable} -- restatement.}
Let $(X,\mathcal C)$ be a separable convexity space with Radon number $D$, and put $d=2^D$ and $h=d+1$. Let $a=a(D)=O(D^3)$ and $C_D$ be as in Theorem~\ref{thm:weak-epsilon-nets}. Then for every $p\ge q\ge h$, every finite family $\mathcal F\subseteq\mathcal C$ of nonempty convex sets satisfying the $(p,q)$-property admits a transversal of size at most
\[
C_D\left(\beta_d\left(\frac{\binom{q}{h}}{\binom{(p-1)(q-1)+1}{h}}\right)\right)^{-a}.
\]

\begin{proof}
Let $p'=(p-1)(q-1)+1$, and set
\[
\alpha=\frac{\binom{q}{h}}{\binom{p'}{h}},
\qquad
\gamma=\beta_d(\alpha).
\]
We first show that for every finite multiset $\mathcal F'$ whose elements are members of $\mathcal F$, there exists a point contained in at least a $\gamma$-fraction of the members of $\mathcal F'$, counted with multiplicity. By multiplying all multiplicities by the same integer if necessary, we may assume that $n=|\mathcal F'|\ge p'$.

The multiset $\mathcal F'$ satisfies the $(p',q)$-property in the following sense. Among any $p'$ members of $\mathcal F'$, either there are $p$ distinct original members of $\mathcal F$, in which case some $q$ of them intersect by the $(p,q)$-property of $\mathcal F$, or one original member appears at least $q$ times, in which case those $q$ copies clearly intersect.

Let $M_h$ be the number of intersecting $h$-tuples of members of $\mathcal F'$. We count pairs $(I,J)$, where $J$ is a $p'$-tuple of members of $\mathcal F'$, and $I\subseteq J$ is an intersecting $h$-tuple. Each $J$ contains at least $\binom{q}{h}$ such $h$-tuples, while each fixed $h$-tuple is contained in at most $\binom{n-h}{p'-h}$ choices of $J$. Therefore
\[
M_h\binom{n-h}{p'-h}
\ge
\binom{n}{p'}\binom{q}{h}.
\]
Dividing by $\binom{n}{h}$ and using
$\binom{n}{p'}\binom{p'}{h}=\binom{n}{h}\binom{n-h}{p'-h}$, we get
\[
\frac{M_h}{\binom{n}{h}}
\ge
\frac{\binom{q}{h}}{\binom{p'}{h}}
=
\alpha.
\]
By the fractional Helly theorem, there is a point $x\in X$ that belongs to at least $\gamma n$ members of $\mathcal F'$, counted with multiplicity.

By the standard linear programming duality lemma used in the Alon--Kleitman proof, for example in the form stated in~\cite[Lemma~2.4]{KellerST17}, it follows that there exists a finite multiset $Y\subseteq X$ such that every $F\in\mathcal F$ contains at least $\gamma |Y|$ points of $Y$, counted with multiplicity.

Apply Theorem~\ref{thm:weak-epsilon-nets} to the multiset $Y$, with $\varepsilon=\gamma$. The proof of the weak $\varepsilon$-net theorem applies to multisets by treating repeated points as labelled copies. We obtain a weak $\gamma$-net $N\subseteq X$ for $Y$ of size at most $C_D\gamma^{-a}$. Since every $F\in\mathcal F$ is convex and contains at least $\gamma |Y|$ points of $Y$, the definition of a weak $\gamma$-net implies that $N\cap F\neq\emptyset$. Thus $N$ is a transversal for $\mathcal F$, and
\[
\tau(\mathcal F)
\le
C_D\gamma^{-a}
=
C_D\left(\beta_d\left(\frac{\binom{q}{h}}{\binom{(p-1)(q-1)+1}{h}}\right)\right)^{-a}.
\]
This completes the proof.
\end{proof}

\begin{remark}
The theorem is stated in terms of the fractional Helly function $\beta_d$, since in this level of generality the dependence of $\beta_d(\alpha)$ on $\alpha$ is part of the quantitative input. 
\end{remark}

\section{A colorful Tverberg theorem for unions of convex sets}\label{sec:s-convex}

We now explain how the same colorful shattering argument yields a colored Tverberg theorem for unions of convex sets. We state and use two auxiliary lemmas from~\cite{AlonS26} in a slightly more general form, suitable for separable convexity spaces. Their proofs are the same as in the Euclidean setting, after replacing Euclidean halfspaces by halfspaces in the convexity space and using Lemma~\ref{lem:vc-radon}.

Let $(X,\mathcal C)$ be a convexity space. A set $F\subseteq X$ is called $s$-convex if it is the union of at most $s$ members of $\mathcal C$. The natural Tverberg-type conclusion for $s$-convex sets is not stated in terms of convex hulls, but rather in the following equivalent spirit: the parts should be chosen so that any $s$-convex sets containing them have a nonempty intersection.

We shall also use the following terminology. A $\mathcal C$-polyhedron with at most $t$ facets is an intersection of at most $t$ halfspaces of $(X,\mathcal C)$. Thus a $\mathcal C$-polyhedron is convex, but need not be bounded in any sense.

\begin{lemma}(\cite[Abstract form of Lemma~3.2]{AlonS26})\label{lem:AS-vc-polytopal-unions}
Let $(X,\mathcal C)$ be a separable convexity space with Radon number $D$. Let $a,t\ge 1$ be integers and put $L=at$. Let $P\subseteq X$ be finite, and let $H=(P,E)$ be the hypergraph in which $S\subseteq P$ is a hyperedge if and only if there exist $\mathcal C$-polyhedra $K_1,\ldots,K_a$, each having at most $t$ facets, such that $S=P\cap\left(\bigcup_{j=1}^a K_j\right)$. Then the VC-dimension of $H$ is bounded by $O(DL\log(L+1))$.
\end{lemma}

\begin{proof}
By Lemma~\ref{lem:vc-radon}, the halfspace hypergraph of $(X,\mathcal C)$ has VC-dimension $D-1$. The family in the statement is obtained from the halfspace family by Boolean formulas consisting of $a$ disjunction terms, each of which is a conjunction of at most $t$ halfspaces. Thus the total number of halfspace occurrences is at most $L=at$. The standard VC-dimension bound for Boolean combinations of a VC class, applied as in \cite[Lemma~3.2]{AlonS26}, gives VC-dimension $O(DL\log(L+1))$.
\end{proof}

\begin{lemma}(\cite[Abstract form of Lemma~4.1]{AlonS26})\label{lem:AS-sconvex-polytopal-enlargement}
Let $(X,\mathcal C)$ be a separable convexity space. Let $q\ge 1$, and let $F_1,\ldots,F_q\subseteq X$ be $q$ sets, each of which is $s$-convex. Assume that $F_1\cap\cdots\cap F_q=\emptyset$. Then there exist sets $K_1,\ldots,K_q$ such that $F_i\subseteq K_i$ for all $i\in[q]$, each $K_i$ is the union of at most $s$ $\mathcal C$-polyhedra, each having at most $s^{q-1}$ facets, and $K_1\cap\cdots\cap K_q=\emptyset$.
\end{lemma}

\begin{proof}
Write each $F_i$ as $F_i=C_{i,1}\cup\cdots\cup C_{i,s}$, where the sets $C_{i,a}$ are convex; if necessary, we repeat $\emptyset$ in order to have exactly $s$ sets. Fix a vector $\alpha=(\alpha_1,\ldots,\alpha_q)\in[s]^q$. Since $F_1\cap\cdots\cap F_q=\emptyset$, we have $C_{1,\alpha_1}\cap\cdots\cap C_{q,\alpha_q}=\emptyset$.

We now enlarge the convex sets $C_{1,\alpha_1},\ldots,C_{q,\alpha_q}$ to halfspaces while preserving the empty total intersection. This is done iteratively. At a given step, suppose that all sets except possibly the $i$-th one have already been replaced by convex supersets, and let $B_i$ be the intersection of these other convex sets. Then $B_i$ is convex and disjoint from $C_{i,\alpha_i}$. By the first separation property, there is a halfspace $H_i^\alpha$ containing $C_{i,\alpha_i}$ whose complementary halfspace contains $B_i$. Replacing $C_{i,\alpha_i}$ by $H_i^\alpha$ therefore keeps the total intersection empty. Repeating this for $i=1,\ldots,q$ gives halfspaces $H_1^\alpha,\ldots,H_q^\alpha$ such that $C_{i,\alpha_i}\subseteq H_i^\alpha$ for every $i$, and $H_1^\alpha\cap\cdots\cap H_q^\alpha=\emptyset$.

For $i\in[q]$ and $a\in[s]$, define $K_{i,a}=\bigcap_{\alpha\in[s]^q:\alpha_i=a} H_i^\alpha$. Then $K_{i,a}$ is a $\mathcal C$-polyhedron with at most $s^{q-1}$ facets, and it contains $C_{i,a}$. Put $K_i=K_{i,1}\cup\cdots\cup K_{i,s}$. Then $F_i\subseteq K_i$, and each $K_i$ has the required form.

It remains to check that $K_1\cap\cdots\cap K_q=\emptyset$. If $x$ belonged to this intersection, then for each $i$ there would be an index $a_i\in[s]$ such that $x\in K_{i,a_i}$. Taking $\alpha=(a_1,\ldots,a_q)$, we would get $x\in H_i^\alpha$ for every $i$, contradicting $H_1^\alpha\cap\cdots\cap H_q^\alpha=\emptyset$. Hence $K_1\cap\cdots\cap K_q=\emptyset$.
\end{proof}

\begin{theorem}[Colorful $k$-wise Tverberg theorem for $s$-convex sets]\label{thm:colorful-tverberg-sconvex}
Let $(X,\mathcal C)$ be a separable convexity space with Radon number $D$. Let $s,k,r$ be positive integers with $r\ge 2$ and $1\le k\le r$, and put $v_{D,s,k}=O(Ds^k\log(s^k+1))$. There exists an absolute constant $C>0$ such that the following holds. Let $C_1,\ldots,C_\ell\subseteq X$ be color classes, each of size $r$, where $\ell\ge Ck v_{D,s,k}\log(kv_{D,s,k}r)$. Then there exists a rainbow partition $C_1\cup\cdots\cup C_\ell=R_1\dot\cup\cdots\dot\cup R_r$ such that for every choice of distinct indices $i_1,\ldots,i_k\in[r]$ and every choice of $s$-convex sets $F_{i_1},\ldots,F_{i_k}\subseteq X$ satisfying $R_{i_j}\subseteq F_{i_j}$ for all $j=1,\ldots,k$, one has $F_{i_1}\cap\cdots\cap F_{i_k}\neq\emptyset$.

Equivalently,
\[
O\!\left(kD s^k\log(s^k+1)\log(kDrs^k\log(s^k+1))\right)
\]
color classes, each of size $r$, suffice.
\end{theorem}

\begin{proof}
Let $S=C_1\cup\cdots\cup C_\ell$. We define an auxiliary hypergraph $H=H_{s,k}(S)$ on the vertex set $S$. A subset $E\subseteq S$ is a hyperedge of $H$ if there exists a set $U\subseteq X$ such that $E=S\cap U$, where $U$ is the union of at most $s$ $\mathcal C$-polyhedra, each having at most $s^{k-1}$ facets.

By Lemma~\ref{lem:AS-vc-polytopal-unions}, applied with $a=s$ and $t=s^{k-1}$, the VC-dimension of $H$ is at most $v_{D,s,k}=O(Ds^k\log(s^k+1))$.

We now apply the colorful shattering bound, Theorem~\ref{thm:intro-colorful-vc-bound}, to this hypergraph. Since $\ell\ge Ck v_{D,s,k}\log(kv_{D,s,k}r)$, the colored set $(C_1,\ldots,C_\ell)$ is not colorfully $(k,r)$-shattered by $H$. Hence there exists a rainbow partition $S=R_1\dot\cup\cdots\dot\cup R_r$ with the following property: for every choice of distinct indices $i_1,\ldots,i_k\in[r]$ and every choice of hyperedges $E_{i_1},\ldots,E_{i_k}$ of $H$ satisfying $R_{i_j}\subseteq E_{i_j}$ for all $j$, we have $E_{i_1}\cap\cdots\cap E_{i_k}\neq\emptyset$.

We claim that this rainbow partition has the desired $s$-convex $k$-wise Tverberg property. Suppose not. Then for some distinct indices $i_1,\ldots,i_k\in[r]$ there exist $s$-convex sets $F_{i_1},\ldots,F_{i_k}$ such that $R_{i_j}\subseteq F_{i_j}$ for all $j=1,\ldots,k$, but $F_{i_1}\cap\cdots\cap F_{i_k}=\emptyset$.

By Lemma~\ref{lem:AS-sconvex-polytopal-enlargement}, applied to the $k$ sets $F_{i_1},\ldots,F_{i_k}$, there exist sets $U_{i_1},\ldots,U_{i_k}$ such that $F_{i_j}\subseteq U_{i_j}$ for every $j$, each $U_{i_j}$ is the union of at most $s$ $\mathcal C$-polyhedra, each having at most $s^{k-1}$ facets, and $U_{i_1}\cap\cdots\cap U_{i_k}=\emptyset$.

For each $j$, set $E_{i_j}=S\cap U_{i_j}$. By the definition of $H$, each $E_{i_j}$ is a hyperedge of $H$, and since $R_{i_j}\subseteq F_{i_j}\subseteq U_{i_j}$, we have $R_{i_j}\subseteq E_{i_j}$. But $E_{i_1}\cap\cdots\cap E_{i_k}\subseteq U_{i_1}\cap\cdots\cap U_{i_k}=\emptyset$, contradicting the choice of the rainbow partition. Therefore no such $s$-convex sets exist, and the theorem follows.
\end{proof}

Taking $k=r$ gives the ordinary colored Tverberg-type statement for $s$-convex sets.

\medskip

\noindent \textbf{Theorem \ref{thm:intro-colorful-tverberg-sconvex} -- restatement: Colorful Tverberg theorem for $s$-convex sets.}
Let $(X,\mathcal C)$ be a separable convexity space with Radon number $D$, and let $s$ and $r\ge 2$ be integers. There is an absolute constant $C>0$ such that the following holds. Let $C_1,\ldots,C_\ell\subseteq X$ be color classes, each of size $r$, where
\[
\ell\ge C rD s^r\log(s^r+1)\log\!\left(r^2D s^r\log(s^r+1)\right).
\]
Then there exists a rainbow partition $C_1\cup\cdots\cup C_\ell=R_1\dot\cup\cdots\dot\cup R_r$ such that for every choice of $s$-convex sets $F_1,\ldots,F_r\subseteq X$ satisfying $R_i\subseteq F_i$ for every $i\in[r]$, one has $F_1\cap\cdots\cap F_r\neq\emptyset$.

\begin{proof}
This is Theorem~\ref{thm:colorful-tverberg-sconvex} with $k=r$.
\end{proof}

\section{A remark on weak separability}\label{sec:weak-separability}

In the main body of the paper, the term separable includes both separation properties stated in the introduction. In this section we drop the first requirement, concerning two disjoint convex sets, and retain only the point--convex-set separation property used by Holmsen and Pat\'akov\'a~\cite{HolmsenP24}: we call a convexity space weakly separable if for every convex set $C\in\mathcal C$ and every point $x\notin C$, there is a halfspace $H$ such that $C\subseteq H$ and $x\notin H$.

In this section we briefly explain how the arguments of the paper can be adapted under this weaker assumption, at the cost of slightly worse dependence on the Radon number. We shall also assume the following compactness property for finitely generated convex sets: for every finitely generated convex set $B=\operatorname{conv}(T)$ and every family of halfspaces $\{H_i\}_{i\in I}$, if $B\cap\bigcap_{i\in I}H_i=\emptyset$, then there exists a finite subfamily $I_0\subseteq I$ such that $B\cap\bigcap_{i\in I_0}H_i=\emptyset$.

Let $D$ be the Radon number, and let $h$ be the Helly number of the space (not to be confused with $h$ from the previous sections that was the fractional Helly number). By Levi's theorem, $h\le D-1$. Under weak separability, the halfspace hypergraph still satisfies $\operatorname{VCdim}(H_{\mathcal C})\le D-1$: if a finite set is shattered by halfspaces, then no partition of it can be a Radon partition. However, the reverse inequality need not hold, and, more importantly, disjoint convex hulls are no longer separated directly by halfspaces.

The replacement is the following simple lemma. Let $S$ be a finite set of points, let $B=\operatorname{conv}(S)$, and let $C_1,\ldots,C_q$ be finitely generated convex sets contained in $B$. Suppose that $C_1\cap\cdots\cap C_q=\emptyset$. For every $y\in B$, choose an index $i(y)$ with $y\notin C_{i(y)}$. By weak separability there is a halfspace $H_y$ such that $C_{i(y)}\subseteq H_y$ and $y\notin H_y$. Hence $B\cap\bigcap_{y\in B}H_y=\emptyset$. By the compactness assumption, there are $y_1,\ldots,y_N\in B$ such that $B\cap H_{y_1}\cap\cdots\cap H_{y_N}=\emptyset$. Applying the Helly theorem to the finite family $B,H_{y_1},\ldots,H_{y_N}$, we may choose at most $h$ of the halfspaces, say $H_{y_j}$ with $j\in J$, such that $B\cap\bigcap_{j\in J}H_{y_j}=\emptyset$. For each $i\in[q]$, let $K_i$ be the intersection of those $H_{y_j}$ with $i(y_j)=i$, and let $K_i=X$ if no such halfspace was chosen. Then $C_i\subseteq K_i$ for every $i$, each $K_i$ is an intersection of at most $h\le D-1$ halfspaces, and $S\cap K_1\cap\cdots\cap K_q=\emptyset$.

Thus, in all shattering arguments one should replace the halfspace hypergraph $H_{\mathcal C}$ by the hypergraph $H_{\mathcal C}^{(\le h)}$ whose hyperedges are intersections of at most $h$ halfspaces. The previous lemma provides exactly the certificates needed for colorful shattering: whenever the relevant convex hulls fail to intersect, one obtains hyperedges of $H_{\mathcal C}^{(\le h)}$ containing the corresponding parts and whose intersection misses the finite point set under consideration.

It remains to estimate the VC-dimension of $H_{\mathcal C}^{(\le h)}$. Since the halfspace hypergraph has VC-dimension at most $D-1$, the number of halfspace traces on an $m$-point set is at most $O(m^{D-1})$ by Perles--Sauer--Shelah. Therefore the number of traces of intersections of at most $h$ halfspaces is at most $O(m^{(D-1)h})$, and the standard comparison with $2^m$ gives $\operatorname{VCdim}(H_{\mathcal C}^{(\le h)})=O(D^2\log D)$. Write $v_D=O(D^2\log D)$ for this new VC parameter.

With this replacement, the proofs of the Tverberg, selection, weak $\varepsilon$-net, and quantitative $(p,q)$ results go through verbatim, with $D$ in the shattering estimates replaced by $v_D$. Consequently, the colorful $k$-wise Tverberg theorem becomes $O(kv_D\log(kv_Dr))$ color classes, each of size $r$. Equivalently, this is $O(kD^2\log D\log(kD^2r\log D))$ colors. Taking $k=D-1$ and applying Levi's theorem gives a colorful Tverberg theorem with $O(D^3\log D\log r)$ color classes of size $r$, for $r>D$.

The uncolored $k$-wise Tverberg theorem then gives a bound of $O(kv_Dr\log(kv_Dr))$ points. In particular, the ordinary Tverberg number satisfies $\operatorname{Tv}_{\mathcal C}(r)\le O(D^3r\log D\log r)$ for $r>D$. Thus one still obtains a bound essentially linear in $r$, but with an additional factor of about $D\log D$ in the dependence on the Radon number.

The colorful selection argument also remains valid. The fractional Helly parameter is unchanged, provided one uses the fractional Helly theorem for the weak notion of separability: we may still take $d=2^D$ and $h_{\mathrm{frac}}=d+1$. Applying the colorful Tverberg theorem with $r=h_{\mathrm{frac}}$ and $k=D-1$ now gives $O(Dv_D\log(Dv_D2^D))=O(D^4\log D)$ colors. Hence the colorful selection lemma and the ordinary selection lemma hold with subset size $a=O(D^4\log D)$.

As a consequence, the weak $\varepsilon$-net theorem holds in the weakly separable setting with weak nets of size $O_D(\varepsilon^{-a})$, where $a=O(D^4\log D)$. The quantitative $(p,q)$-theorem also remains valid with the same change in the exponent: the bound is the same as in Theorem~\ref{thm:pq-separable}, but with $a=O(D^4\log D)$. Finally, the second colored Tverberg tradeoff becomes the following: there are $\ell_0=O(D^4\log D)$ color classes such that, if each color class has size at least $C_Dr$, then one can find $r$ pairwise disjoint rainbow sets whose convex hulls have a common point.

\section*{Acknowledgements}

\paragraph{Use of generative AI.}
The core mathematical ideas, results, and proofs in this paper are due to the authors. The authors used ChatGPT as an auxiliary tool in discussions of computational and technical details and in editing parts of the exposition. 

\bibliographystyle{alpha}
\bibliography{references-isf}

\end{document}